\documentclass[a4paper,10pt]{amsart}
\usepackage[utf8]{inputenc}
\usepackage[english]{babel}

\usepackage{amsfonts}

\usepackage[english]{babel}
\usepackage{amsmath,amssymb,verbatim,mathrsfs,latexsym,paralist,graphics, enumerate}

\usepackage{amsfonts,mathtools,url, enumerate} 
\usepackage{subcaption}
\usepackage{float}
\usepackage{tikz-cd}
\usepackage{hyperref}
\usepackage{indentfirst}
\usepackage{multirow}
\usepackage[a4paper]{geometry}
\usepackage{amsthm}
\allowdisplaybreaks[4]
\usepackage{soul}
\usepackage{ tipa }
\usepackage{tikz}

\usepackage{enumitem}
\usepackage{url} 
\usepackage[all]{xy}

\newcommand{\bC}{\mathbf{C}}
\newcommand{\bD}{\mathbf{D}}

\newcommand{\Mat}{\mathbf{Mat}}
\newcommand{\Graph}{\mathbf{Graph}}

\newcommand{\LF}{\mathbf{IMat}}

\newcommand{\Ind}{\mathrm{Ind}}

\DeclareMathOperator{\Hom}{\mathrm{Hom}}

\newcommand{\Z}{\mathbb{Z}}

\newcommand{\N}{\mathbb{N}}

\newcommand{\mrm}{\mathrm}

\newcommand{\mc}{\mathcal}

\newcommand{\ob}{\mathrm{ob}}
\newcommand{\Id}{\mathrm{Id}}
\newcommand{\cS}{\mathcal{S}}
\newcommand{\cF}{\mathcal{F}}
\newcommand{\cC}{\mathcal{C}}
\newcommand{\tc}{\mathrm{tc}}

\newcommand{\bigslant}[2]{{\raisebox{.2em}{$#1$}\left/\raisebox{-.2em}{$#2$}\right.}}

\numberwithin{equation}{section}
\newtheorem{theorem}{Theorem}[section]
\newtheorem{thm}{Theorem}
\newtheorem{proposition}[theorem]{Proposition}

\newtheorem{lemma}[theorem]{Lemma}
\newtheorem{corollary}[theorem]{Corollary}
 
\newtheorem{defn}[theorem]{Definition}    
\newtheorem{question}[theorem]{Question}
\theoremstyle{remark}
 
\theoremstyle{remark}
\newtheorem{rmk}[theorem]{Remark}
\newtheorem{example}[theorem]{Example}

\definecolor{bunired}{rgb}{0.8, 0.0, 0.0}
\definecolor{caribbeangreen}{rgb}{0.0, 0.8, 0.6}

\title{On the K-theory of  matroids with Tutte coverings}
\author{Luigi Caputi}
\address{Dipartimento di Matematica, Universit\`a di Bologna, via Zamboni 33, Bologna, IT}
\email{luigi.caputi@unibo.it}
\author{Sabino Di Trani}
\address{Dipartimento di Ingegneria e Scienze dell' Informazione e Matematica, Università degli Studi dell'Aquila, Via Vetorio, L'Aquila, IT}
\email{sabino.ditrani.math@gmail.com}
\date{}

\begin{document}

\maketitle

\begin{abstract}
    The aim of this work is to explicitly compute the $K$-theory of the category of matroids with  respect to the covering family of Tutte coverings. In particular, we show that this is equivalent to the $K$-theory spectrum of the category of graphic matroids on looped forests, with the covering family generated by  isomorphisms.  Further, we show that this yields an equivalence of $C_2$-spectra.
\end{abstract}

\section*{Introduction}

 Matroids provide a combinatorial framework that unifies several notions of independence arising in algebra, geometry, and graph theory. The most common numerical and algebraic invariants of matroids satisfy recursive descriptions in terms of deletion and contraction operations. Tutte-Grothendieck invariants play a fundamental role in this regard, providing universal invariants compatible with the deletion-contraction relations. In fact, the Tutte polynomial is a central object in matroid theory, encoding several important specializations, including the chromatic polynomial, the flow polynomial, and the characteristic polynomial (see \cite[Chapter 6]{WhiteMatApp} for a complete exposition). 

In recent years, understanding  polynomial invariants related to the Tutte polynomial from a categorical and homotopical perspective has gained increasing attention. An example  is given by the process of categorification of the Tutte polynomial of a graph, that describes such a polynomial as a graded Euler characteristic~\cite{jasso2006categorification}. Among others, there are also categorifications of the chromatic polynomial~\cite{Helme-Guizon-Rong}, of the dichromatic polynomial~\cite{05313438}, and of the signed chromatic polynomial of graphs~\cite{zbMATH07555134}. Related ideas appear in recent approaches to categorify classical matroid invariants, such as the characteristic polynomial~\cite{saito2024categorificationcharacteristicpolynomialmatroids}. Another direction is to use homotopic theoretic methods. The framework of categories with covering families, introduced in~\cite{zbMATH07985809}, has provided a natural setting for defining a $K$-theory spectrum associated to combinatorial structures as graphs, posets or matroids.  Recent developments in this direction include the work~\cite{calle2024combinatorial} on the edge reconstruction conjecture in graph theory, and the study of categorical realizations of Tutte-type recursions~\cite{lopez2025realizing}, which both inspired this work. 

The deletion--contraction recursion satisfied by the Tutte polynomial  suggests a natural covering family structure on the category of matroids. Roughly speaking, a matroid $M$ can be covered by the pair of minors $M\setminus e$ and $M/e$ associated to a non-degenerate element $e$. Iterating this process produces a deletion--contraction tree whose leaves correspond to matroids obtained by successive deletions and contractions. Following~\cite{lopez2025realizing}, these trees define the \emph{Tutte coverings} of a matroid and determine a covering family structure. This yields a category with covering families, denoted $\Mat^{\tc}$.
The goal of this paper is to compute the  $K$-theory spectrum associated to this category. Our main result is the following:
\begin{thm}[cf.~Corollary~\ref{cor:splitting}]\label{thmintro}
      There is an equivalence of spectra
    \[
    K(\Mat^{\tc})\simeq \bigvee_{m,n\in\mathbb{N}} \Sigma_+^{\infty} B(S_m)\wedge \Sigma_+^{\infty} B(S_n)
    \ , \]
    where $S_p$ denotes the symmetric group on $p$ elements.
\end{thm}

The key observation is that every matroid admits a Tutte covering whose domains are \emph{indecomposable} matroids (see Definition~\ref{def:idencmatroids}). Such matroids have a particularly simple structure: they are precisely the direct sums of loops and coloops, and can be identified with graphic matroids associated to looped forests~\cite{lopez2025realizing}. 
Using this description, we show that the category of matroids with Tutte coverings  $\Mat^{\tc}$ has refinements (see Proposition~\ref{cor:refinements}). This allows us to apply the D\'evissage theorem for categories with covering families of~\cite{calle2024combinatorial}. As a consequence, the computation of the $K$-theory of $\Mat^{\tc}$ can be reduced to the computation of the $K$-theory of the full subcategory of indecomposable matroids, with respect to the covering family given by isomorphisms of matroids. This gives Theorem~\ref{thmintro}.

In view of Theorem~\ref{thmintro}, the recursive combinatorics underlying the Tutte polynomial  investigated in~\cite{lopez2025realizing} translates into a homotopy-theoretic description of the associated $K$-theory spectrum. We further show that the equivalence on $K$-theory is an equivalence of $C_2$-spectra (see Corollary~\ref{cor:dual}) and that this can be done also for graphic matroids, seen as a full category of $\Mat^{\tc}$ (see Theorem~\ref{thm:iso}).

\subsection*{Acknowledgements}
LC~was supported by the Starting Grant 101077154 ``Definable Algebraic Topology'' from the European Research Council of Martino Lupini. SDT is supported by the PRIN Project 20223FEA2E - ``Cluster algebras and Poisson Lie groups". The authors acknowledge partial support from INdAM-GNSAGA. Furthermore, the authors would like to sincerely thank the organizers and  all the participants (frogs included) in the conference ``Collaborations in Algebra, Representation Theory, and Ethics'' in Lyon, where this work started: thank you for reminding us what a good idea it is to take CARE of each other. 

 \section{Matroids and their Algebraic Invariants}
In this sectionwe recall some basic facts of matroids needed in the follow-up. We refer to \cite{OxMat} and \cite{WhiteMaTh} for  complete references.
 
Recall taht a (finite) \emph{matroid} $M$ is a pair $(X, \mc{I})$, where $X$ is a (finite) set and $\mc{I}$ is  a family of independent sets for~$X$, i.e.~a family of subsets of $X$ that satisfies the independence axioms: 
\begin{enumerate}[label={\rm (I\arabic*)},itemsep=0.11cm]
\item  the empty set is in $\mc{I}$;
\item  if $A\in \mc{I}$ and $B \subset A$, then $B \in \mc{I}$;
\item \label{I3}
if $A,B \in \mc{I}$ and $|A| > |B|$ then there exists $x \in A$ such that $B \cup \{x\} \in \mc{I}$.
\end{enumerate}
The set $X$ is the \emph{ground set} of~$M$. A maximal independent set is called a \emph{basis}. 
  
Given a matroid  $M=(X,\mc{I})$, its \emph{rank function}  $r_M\colon \mc{P}(X) \rightarrow \mathbb{N}_{\geq 0}$ is the function which associates to each $S\subseteq X$ the size of a maximal independent set contained in $S$. 

\begin{example}[Uniform Matroids]\label{ex:uniform}
Let $X$ be a finite set with $n$ elements and $k \leq n$ a non negative integer. The family 
$\mc{I}=\{ I \subset X \, | \, |I|\leq k \}$ satisfies the axioms of a family of indipendents. The matroid $M=(X, \mc{I})$ is the uniform matroid of subsets of cardinality $k$ in a set with $n$ elements. It is denoted by $U_{k,n}$. 
\end{example}

\begin{example}[Graphic Matroids]\label{ex:graphic}
Let $G$ be a finite undirected graph, and $\mc{C}$ the set of its (undirected) cycles. Identifying a subgraph of $G$ with the set of its edges, the set 
 \[ \mc{I}_{\mc{C}} = \{ S \subseteq E(G) \mid C \nsubseteq S,\text{ for each }C\in\mc{C} \}\ \]  satisfies the independence axioms, where $E(G)$ denotes the set of edges of $G$,.
The pair $M_G=(E(G), \mc{I}_{\mc{C}})$ is the graphic matroid associated to $G$.
\end{example}

Let $M_1=(X_1,\mc{I}_1)$ and $M_2=(X_2,\mc{I}_2)$ be two matroids. A matroid morphism $\varphi\colon M_1 \rightarrow M_2$ is an injective map $\varphi\colon X_1 \rightarrow X_2$ such that $\varphi(I) \in \mc{I}_2$ for every $I\in \mc{I}_1$. Matroids and morphisms of matroids yield a category that we denote by $\Mat$.
The matroids $M_1=(X_1, \mc{I}_1)$ and $M_2=(X_2, \mc{I}_2)$ are isomorphic, and in such case we write $M_1\cong M_2$, if there exists a bijective map $f \colon X_1 \rightarrow X_2$ such that  $f^{-1}(I) \in \mc{I}_1$ if and only if $I \in \mc{I}_2$.
\begin{defn}[Direct Sum Matroid]\label{defn:summat} Let $M_1=(X_1, \mc{I}_1)$ and $M_2=(X_2, \mc{I}_2)$ be  matroids. The direct sum $M_1 \oplus M_2$ is the matroid having ground set $X_1 \sqcup X_2$ and $\mc{I} = \{ I_1 \sqcup I_2 \mid I_1\in \mc{I}_1, I_2 \in \mc{I}_2\}$ as family of independent sets.
\end{defn}

Let $M=(X, \mc{I})$ be a matroid and $B$ the set of its bases. 
Set 
\[B^*=\{ X \setminus J \,|\, J \in B\} \ , \]
\[I^*=\{A \in X \, | \, A \subset J \mbox{ for certain } J \in B^*  \}\ .\]
Observe that the set $I^*$ satisfies the axioms of family of indipendent sets for a matroid.

\begin{defn}[Dual Matroid]\label{defn:dual}  Let $M=(X, \mc{I})$ be a matroid. The matroid $M^*=(X,I^*)$ is the \emph{dual matroid} of $M$.
\end{defn}

We will use the following terminology;
\begin{defn}\label{def:terminologye}
Let $M=(X,I)$ be a finite matroid and $e\in X$. 
\begin{itemize}
    \item $e$ is a \emph{loop}  if $\{e\}$ is not an independent set;
    \item $e$ is a \emph{coloop} (or an isthmus) of $M$ if $e$ is contained in every independent set;
    \item $e$ is \emph{non-degenerate} if it is not a loop nor a coloop for $M$. 
\end{itemize}  
\end{defn}

There are two  operations on matroids of importance in the follow-up: deletion and contraction. We  now recall the definitions.
Consider a matroid $M=(X,\mc{I})$ and $T$ any subset of $X$.
\begin{defn}[Deletion Matroid]\label{defn:deletionmat}The deletion of $M=(X,\mc{I})$ with respect to $T$ is the matroid $M \setminus T=(X', \mc{I'})$ defined by: 
\begin{enumerate}[label={\rm (DM\arabic*)},itemsep=0.11cm]
\item the ground set $X'$ is the set $X \setminus T$; 
\item the independent set $\mc{I'}$ is the sets of $I \in \mc{I}$ such that $I \subset X'$.
\end{enumerate}
\end{defn}

\begin{defn}[Contraction Matroid]\label{defn:contractionmat} The matroid $M / T$ is defined as the matroid $(M^* \setminus T ) ^*$
\end{defn}
\begin{rmk}\label{rmk:usualcontraction}
If $\{x\}$ is an independent for $M=(X,\mc{I})$ it is equivalent to say that $M / x$ is the matroid defined by the data: 
\begin{enumerate}[label={\rm (CM\arabic*)},itemsep=0.11cm]
\item the ground set $X$ is the set $X \setminus \{x\}$; 
\item the independent set $\mc{I}'$ is given by the sets $I \subset X \setminus \{x\}$ such that   $I \cup \{x\} \in \mc{I}$. 
\end{enumerate}
Otherwise, if $x$ is a loop for $M$, then $M/x=(X \setminus \{x\}, \mc{I})$. \end{rmk}

We recall the following property of deletions and contractions  -- see, e.g.~\cite[Proposition 3.1.25]{OxMat}: 
\begin{proposition}\label{prop:minorsareWD}
Let $M=(X, I)$ be a matroid and $T_1$, $T_2$ disjoint subsets of $X$. Then:
\begin{enumerate}
\item $(M \setminus T_1) \setminus T_2 = M \setminus (T_1 \cup T_2)= (M \setminus T_2) \setminus T_1$;
\item $(M / T_1) / T_2 = M / (T_1 \cup T_2)= (M / T_2) / T_1$;
\item $(M \setminus T_1) / T_2 =  (M / T_2) \setminus T_1$.
\end{enumerate}

\end{proposition}

    Let $M$ be a matroid. Any matroid obtained from $M$ via a sequence
of deletions and/or contractions is called a \emph{minor} of $M$.
 Observe that, in view of Proposition~\ref{prop:minorsareWD} the minors of a matroid are well defined objects.

\begin{rmk}\label{rmk:DCmorphisms} Let $M=(X,\mc{I}(M))$ be a matroid and $x \in X$. The inclusion of $X \setminus \{x\}$ induces two maps, 
\[ M \setminus x \rightarrow M \  \qquad \text{and} \qquad M / x \rightarrow M \ ,\]
that are matroid morphisms. In fact, for the first map, by definition we have $I \in \mc{I}(M \setminus x)$ if and only if $I$ is an independent set for $M$ contained in $X \setminus \{x\}$. The second map  is obviously defined if $x$ is a loop. If $x$ is not a loop, we have $I \in \mc{I}(M/x)$ if and only if $I \cup \{x\} \in \mc{I}(M)$; in particular $I \in \mc{I}(M)$. Furthermore, observe that if $T_1$ and $T_2$ are disjoint subset of $X$, then the ground set inclusion induces a morphism from $M'=M \setminus T_1 / T_2 $ to $M$.
\end{rmk}

For the class of graphic matroids, we have the following:
 \begin{proposition}[{\cite[3.1.2 \& 3.2.1]{OxMat}}]\label{prop:closuregraphs}
    Let $G$ be a graph and $M_G$ be the associated graph matroid. Let $T\subseteq E(G)$ a subset of edges of $G$. Then:
    \begin{enumerate}
        \item $M_G\setminus T=M_{G\setminus T}$;
        \item $M_G/T=M_{G/T}$.
    \end{enumerate}
    where $G\setminus T$ and $G/T$ denotes the classical deletion and contraction operations on graphs.
\end{proposition} 

By Proposition~\ref{prop:closuregraphs}, minors of graphic matroids are graphic matroids as well.

Let $M = (X,\mc{I})$ be a matroid with rank function $r$.
The \emph{corank} of a subset $A \subset X$ is the difference $z(A)\coloneqq r(M) -r(A)$. The difference $n(A)\coloneqq |A|-r(A)$ is called the \emph{nullity} of $A$. 
Using the corank and the nullity functions, we can now recall the definition of Tutte polynomial for a matroid. For more details we refer to \cite[Chapter~6]{WhiteMatApp}. 

\begin{defn}\label{def:Tutte} Let $M=(X,I)$ be a finite matroid. Then, its Tutte polyonomial is defined as: 
\[T_M(x,y) = \sum_{A \subseteq X} (x-1)^{z(A)}(y-1)^{n(A)} \ .\]
By convention, if $X = \emptyset $ then $T_M(x,y) = 1$.
\end{defn}

We recall that the Tutte polynomial  is well-behaved with respect to direct sums, 
\begin{equation}\label{prop:DSTutte}
    T_{M_1 \oplus M_2}(x,y)=T_{M_1}(x,y) \, T_{M_2}(x, y) \ ,
\end{equation}
and it can be recursively computed using  deletions and contractions. These properties of the Tutte polynomial are essentially the main motivation for investigating the category of matroids with Tutte coverings introduced in \cite{lopez2025realizing}, and that we are going to study in this work.

 \section{$K$-theory of categories with covering families}

 In this section we recall the notion of a category with covering families.
We follow~\cite{zbMATH07985809}.

\begin{defn}[{\cite[Definition~2.1]{zbMATH07985809}}]
Let $\bC$ be a small category. 
\begin{enumerate}
    \item A \emph{multi-morphism} in $\bC$ is an object $B\in\bC$ and a finite (possibly empty) family of maps in $\bC$ with target $B$:
    \[
    \{f_i\colon A_i\to B\}_{i\in I} \ .
    \]
  \item   A \emph{covering family structure } on $\bC$ is a collection of multi-morphisms, called the \emph{covering families},
such that:
\begin{enumerate}
    \item each singleton containing an identity map  $\{A=A\}$ is a covering family;
    \item given a covering family $\{g_j\colon B_j\to C\}_{j\in J}$ and, for each $j\in J$ a covering family $\{f_{ij}\colon A_{ij}\to B_j\}_{i\in I_j}$, then the collection of all composites $\{g_j\circ f_{ij}\colon A_{ij}\to C\}_{j\in J, i\in I_j}$ forms a covering family.
\end{enumerate}
\item A  \emph{category with covering families} is a small category $\bC$, a covering family structure on $\bC$, and a
distinguished basepoint object $\ast\in \bC$, such that:
\begin{enumerate}[label=(\roman*)]
    \item $\Hom(\ast,\ast)=\{1_\ast\}$ and $\Hom(c,\ast)=\emptyset$ if $c\neq \ast$;
    \item For every finite (possibly empty) set $I$, the family $\{\ast\to\ast\}_{i\in I}$ is a covering family.
\end{enumerate}
\end{enumerate}   
\end{defn}

A morphism of categories with covering families is a functor $F\colon \bC\to \mathbf{D}$ that preserves both the covering families and the basepoint; this means that $F(\ast_C)=\ast_D$ and $F$ sends covering families of $\bC$ to covering families of $\bD$. Categories with covering families and morphisms of categories with covering families yield a category $\mathbf{CatFam}$.

If $\bC$ is a small category without basepoint, we can always add one as \emph{disjoint base point}. In fact, from $\bC$ we can form a new category $\bC_+$  by adding an object $\ast$ to $\bC$ with the property that $\Hom(\ast,\ast)=\{1_\ast\}$ and $\Hom(c,\ast)=\emptyset$ if $c\neq \ast$. This new category $\bC_+$ becomes a category with covering families if  all families of the form $\{\ast\to\ast\}_{i\in I}$
are covering families~\cite[Remark~2.3]{zbMATH07985809}.

\begin{example}\label{ex:catisos}
    Let $\bC$ be a small category with basepoint $\ast$. Then, the collection $\mathcal{S}_\bC^{\cong}$ of all isomorphisms of $\bC$ yields a category with covering families that we denote by $(\bC,\mathcal{S}_\bC^{\cong},\ast)$ or, more concisely by $\bC^{\cong}$.  In particular, we have:
    \begin{enumerate}[label=(\roman*)]
        \item for every finite (possibly empty) set $I$, the family $\{\ast\to\ast\}_{i\in I}$ is in $\mathcal{S}^{\cong}$;
        \item if $f\colon C\to D$ is an isomorphism in $\bC$, then $\{f\colon C\to D\}$ is in $\mathcal{S}^{\cong}$.
    \end{enumerate}
\end{example}

\begin{example}[{\cite[Definition~4.1]{lopez2025realizing}}]\label{ex:matroids}
    Let $\Mat$ be the category of matroids. Denote by $\Mat_+$ the category of matroids and matroids morphisms, with a disjoint distinguished point $\ast$. Then, the category with covering families  $(\Mat_+,\mathcal{S}_\Mat^{\cong},\ast)$ shall be denoted by $\Mat^{\cong}$.
\end{example}

We are interested in the $K$-theory of categories with covering families, and in particular in the $K$-theory of the category of matroids. Its definition is based on the following  notion of categories of covers:

\begin{defn}[{\cite[Definition~2.12]{zbMATH07985809}}]
    Let $\bC$ be a category with covering families. The \emph{category of covers} $W(\bC)$ has objects any finite collection $\{A_i\}_{i\in I}$ of non-distinguished objects of~$\bC$, and a morphism
    $
    \{A_i\}_{i\in I}\to \{B_j\}_{j\in J}
    $
    in $W(\bC)$  is a map $f\colon I\to J$ of finite sets along with morphisms $f_i\colon A_i\to B_{f(i)}$ for all $i\in I$, such that 
    \[
    \{f_i\colon A_i\to B_j\}_{i\in f^{-1}(j)}
    \]
    is a covering family for each $j\in J$. Composition of morphisms in $W(\bC)$ is defined by composing the maps of sets and the covering families. 
\end{defn}

We observe here that the category of covers construction yields a functor by \cite[Lemma~2.16]{zbMATH07985809}, and that $W(\bC)$ has a natural base-point object, corresponding to $I=\emptyset$. Using the category of covers associated to a category with covering families~$\bC$, it is possible to define the $K$-theory spectrum $K(\bC)$ -- see  \cite[Definition~2.17]{zbMATH07985809} for the details of this construction. The computation of the $K$-theory groups is in general quite difficult, but the case of the $0$-th group is well understood. In fact, we have the following:

\begin{proposition}[{\cite[Proposition~3.8]{zbMATH07985809}}]
    Let $\bC$ be a category with covering families. Then, $K_0(\bC)$ is the free abelian group
    \[
    K_0(\bC)=
    \bigslant{\Z[\ob(\bC)]}{[A]=\sum_{i\in I}[A_i]}
    \]
for every covering family $\{A_i\to A\}_{i\in I}$.
\end{proposition}

The $K$-theory of categories with covering families satisfies fundamental theorems valid also for the classical $K$-theory of categories.  We let $\bD\subseteq \bC$ be a full subcategory of the category with covering families~$\bC$; that is, $\bD$ is a full subcategory of $\bC$, and $\bD\to\bC$ is a morphism of categories with covering families. Following \cite{calle2024combinatorial}, we say that a cover $\{f_i\colon A_i\to A\}_{i\in I}$ is a \emph{refinement} of a cover $\{g_j\colon B_j\to A\}_{j\in J}$ if for every $i\in I$ there is $j\in J$ and a map $h\colon A_i\to B_j$ in $\bC$ 
such that $f_i=g_j\circ h$. 

\begin{defn}[{\cite[Definition~2.7]{calle2024combinatorial}}]
A category with covering families $\bC$ has refinements if any pair of covers $\{A_i\to A\}_{i\in I}$ and $\{A_j'\to A\}_{j\in J}$ of the same object $A$ has a common refinement. 
\end{defn}

We can now recall the D\'evissage Theorem for categories with covering families.

\begin{theorem}[{\cite[Theorem~2.8]{calle2024combinatorial}}]\label{thm:devissage}
   Let  $\bC$ be a category with covering families that has refinements and let $i\colon\bD\to \bC$ be the inclusion of a full subcategory with covering families. 
   If every object $C\in \bC$ has a choice of a cover $\{D_i\to C\}_{i\in I}$ with $D_i\in \bD$ for all $i\in I$, then
   \[
   Ki\colon K(\bD)\to K(\bC)
   \]
   is an equivalence. 
\end{theorem}

Recall from Example~\ref{ex:catisos} that to any small category with basepoint $\bC$ we can associate the category with covering families  $\bC^{\cong}$ where the  covering families are given by the isomorphisms. Then, we have a complete description of the associated $K$-theory:

\begin{proposition}[{\cite[Proposition~2.14]{calle2024combinatorial}}]\label{prop:isos}
    Let $\bC^{\cong}$ be a small category with basepoint. Then, there is an equivalence of spectra
    \[
    K(\bC^{\cong})\simeq \bigvee_{[C]}\Sigma_+^{\infty} B(\mathrm{iso}_{\bC}(C))
    \]
    where the wedge ranges across the isomorphism classes of objects in $\bC$ and $\mathrm{iso}_{\bC}(C)$ is the
group of $\bC$-isomorphisms of an object $C  \in \bC$.
\end{proposition}

As a consequence, Proposition~\ref{prop:isos} gives us the full computation of the category of matroids $\Mat^{\cong}$ with coverings given by the isomorphisms. In the next sections we shall focus on a different structure of covering families, the Tutte coverings developed in \cite{lopez2025realizing}.

\section{Matroids with Tutte coverings}

In this section we first recall from \cite{lopez2025realizing} the definition of the covering category of matroids with Tutte coverings, and then we shall focus on indecomposable matroids. However, to define the notion of Tutte coverings we first need to recall the definition of deletion-contraction trees.

Let $T$ be a binary tree with root $r_T$. We can naturally associate a category  to $T$ as follows. First, we consider the order relation $\leq_T$ on the vertices of $T$ with $w\leq_T v$ if and only if $w=v$ or $w$ is a descendant of $v$. As $T$ is a tree, this describes a poset, hence   a category. We denote this resulting category by $\cC_T$.   

Recall from Example~\ref{ex:matroids} that $\Mat_+$ denotes the category of matroids and matroids morphisms, with a disjoint distinguished point $\ast$.

\begin{defn}[{\cite[Definition~4.4 \& 4.5]{lopez2025realizing}}]\label{defn:DCtree}
    Let $T$ be a rooted binary tree with root~$r_T$. An \emph{elementary deletion-contraction tree of shape $T$} is a functor $\cF\colon \cC_T\to \Mat_+$ such that:
    \begin{enumerate}
        \item $\cF(r_T)\neq \ast$;
        \item\label{def:DCtreeitem2} if $v$ is an internal vertex of $T$ with only one child $w$, then $\cF(v)=\cF(w)$ and if $w\to v$ is the morphism in $\cC_T$ induced by the relation $w\leq_T v$, we have $\cF(w\to v)=\Id_{\cF(v)}$;
        \item if $v$ is an internal vertex with children $w_1$ and $w_2$, then there is a non-degenerate element~$e$ of the matroid $M\coloneqq \cF(v)$ for which $\cF$ maps the subdiagram
         
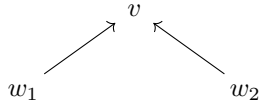
\begin{center}
 \begin{tikzcd}
 & v & \\
 w_1\arrow[ur, ""]  & &w_2 \arrow[ul, ""] 
 \end{tikzcd}    
 \end{center}
 in $\cC_T$ to one of the following ones:
          
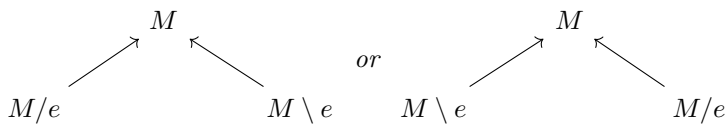
\begin{center}
 \begin{tikzcd}
 & M & \\
 M/e\arrow[ur, ""]  & &M\setminus e \arrow[ul, ""] 
 \end{tikzcd}    
 or
  \begin{tikzcd}
 & M & \\
 M\setminus e\arrow[ur, ""]  & &M/ e \arrow[ul, ""] 
 \end{tikzcd}    
 \end{center}
 where $M/e$ and $M\setminus e$ denote the contraction and deletion of $M$ on $e$, and the maps $M/e\to M$ and $M\setminus e\to M$ denote the standard inclusions. 
    \end{enumerate}
    We say that  a functor $\mathcal{G}\colon \cC_T\to \Mat_+$ is a \emph{deletion-contraction tree} of shape $T$ if it is naturally isomorphic to  an elementary deletion-contraction tree of shape~$T$. 
\end{defn}

Following \cite[Definition~4.11]{lopez2025realizing}, we can now define the covering family structure on the category of matroids induced by the Tutte morphism.

\begin{defn}\label{def:tc}
Let $\cS^{\tc}$ be the collection of multi-morphisms on $\Mat_+$ defined as follows:
\begin{enumerate}
    \item For every finite (possibly empty) set $I$, the family $\{\ast\to\ast\}_{i\in I}$ is in $\cS^\tc$.
    \item \label{def:tcitem2} Given a matroid $M\neq \ast$, a multi-morphism $\{f_j\colon M_j\to M\}_{j=1,\dots,p}$ is in $\cS^\tc$ if and only if there exists a rooted binary tree $T$ with root $r_T$, leaves $v_1,\dots,v_p$ and
a deletion-contraction tree $\cF\colon \cC_T\to\Mat_+$, with $\cF(r_T)=M$, such that $f_j=\cF(i_{v_j})$ for each $j=1,\dots, p$, and where $i_{v_j}$ is the unique morphism $i_{v_j}\colon v_j\to r_T$ in $\cC_T$. 
\end{enumerate}
A multi-morphism $\{f_j\colon M_j\to N\}_{j=1,\dots,p}$ in  $\cS^{\tc}$  is called a \emph{Tutte covering} of $M$. 
\end{defn}

By \cite[Proposition~4.12]{lopez2025realizing}, the triple $(\Mat_+,\mathcal{S}^{\tc},\ast)$ is a category with covering families that we shall shortly denote by $\Mat^\tc$.

\subsection{Indecomposable Coverings}
In this subsection we recall  the properties of the so-called \emph{indecomposable matroids}. Recall from \cite[Definition~2.9]{lopez2025realizing} that, if $\bC$ is a category with covering families, then an object $C\in \bC$ is \emph{indecomposable} if the only coverings of $C$ are singletons of the form $\{A\xrightarrow{\cong} C\}$.  

\begin{defn}\label{def:idencmatroids} An \emph{indecomposable matroid} is a matroid $M=(X,\mc{I})$ such that every $x \in X$ is degenerate. 
\end{defn}

Observe that an indecomposable matroid is the graphic matroid associated to a forest,  possibly with loops. In fact, $M$ is an indecomposable matroid if it is the direct sum (possibly empty) of finitely many coloops
and loops. As by \cite[Definition 4.15]{lopez2025realizing}, we say that a Tutte covering $\{A_i \rightarrow M\}_{i\in I}$ for a matroid $M\neq\ast$  is \emph{indecomposable} if all the matroids $A_i$, for $i\in I$, are indecomposable.
 Every matroid admits a covering by indecomposable. In fact, we have the following:

\begin{proposition}[{\cite[Proposition 4.14]{lopez2025realizing}}]\label{prop:exindec}
Any matroid $M \neq *$ admits a Tutte covering $\{g_i\colon N_i \rightarrow M\}_{i\in I}$ where each $N_i$ is an indecomposable matroid.
\end{proposition}

By a multi-set of matroids we shall  mean an indexed family of matroids $\{M_i\}_{i\in I}$ with $I\neq \emptyset$. An isomorphism of multi-sets of matroids $\{M_i\}_{i\in I}\to \{N_j\}_{j\in J}$ is a bijection $f\colon I\to J$ such that $M_i\xrightarrow{\cong}N_{f(i)}$ is an isomorphism of matroids for all $i\in I$.
Given an indecomposable Tutte covering $A=\{A_i \rightarrow M\}_{i\in I}$, the set $\Ind_M(A)$ will denote the multi-set of domains -- that is, indecomposable matroids -- appearing in $A$. Then, it is proven in \cite[Proposition 4.16]{lopez2025realizing} that indecomposable Tutte coverings have isomorphic domains. More precisely, we have:

\begin{proposition}[{\cite[Proposition 4.16]{lopez2025realizing}}]\label{prop:indareiso}
Let $A\coloneqq\{A_i \rightarrow M\}_{i\in I}$ and $B\coloneqq\{B_j \rightarrow M\}_{j\in J}$ be indecomposable Tutte coverings of a matroid $M\neq\ast$. Then $\Ind_M(A)$ and $\Ind_M(B)$ are isomorphic as multi-sets.
\end{proposition}

As a main step towards the computation of the $K$-theory of matroids with Tutte coverings, we want to show that the category with covering families $\Mat^\tc$ has refinements. 
We will use the following proposition;

\begin{proposition}\label{prop:indecrefinement} Let $A\coloneqq\{A_i \rightarrow M\}_{i\in I}$ be a Tutte covering. Then there exist a refinement $A'\coloneqq\{A'_k \rightarrow M\}_{k\in K}$ of $A$ that is an indecomposable Tutte covering.
\end{proposition}

\proof  If every domain~$A_i$ of $A$ is an indecomposable matroid, the assertion is immediate. Otherwise, there exists an index $j$ and a matroid $A_j$ such that $A_j$ has at least a non-degenerate element $e$.
By Remark~\ref{rmk:DCmorphisms}, we have the morphisms 
\[f_j^d\coloneqq  A_j \setminus e \rightarrow A_j \rightarrow M, \qquad f_j^c\coloneqq A_j / e \rightarrow A_j \rightarrow M\ , \]
and, consequently, we can consider the multi-morphism 
\[ A'\coloneqq\{A_i \rightarrow M\}_{i \neq j} \cup \{ A_j \setminus e \xrightarrow{f_j^d} M\} \cup \{ A_j / e \xrightarrow{f_j^d} M\}\ .\]
We want to show that $A'$ is a Tutte covering family. To see it, let $\cF_A\colon T_A \rightarrow \Mat_+$ be a deletion-contraction tree for the covering family $A$ and let $\ell_j$ be a leaf of $T_A$ such that $\cF_A(\ell_j)=A_j$. Consider the tree $T_{A'}$ obtained from $T_A$ by adding  the children $w_1$ and $w_2$ to $\ell_j$. 
The assignment 
\[ \cF_{A'}(v)\coloneqq\begin{cases} \cF_A(v) &\mbox{ if } v \in T_A \\ M \setminus e &\mbox{ if } v=w_1 \\ M /e &\mbox{ if } v=w_2 \end{cases}  \]
provides the desired functor $\cF_{A'}\colon T_{A'} \rightarrow \Mat_+$. 
By construction,  $A'$ refines $A$. We can now iterate this process until every matroid in the domain of the newly constructed Tutte covering has no non-degenerate elements. This proves the statement.
\endproof 
\begin{corollary}\label{cor:refinements}
 The category with covering families $\Mat^\tc$ has refinements.
\end{corollary}
\proof 
Let $\{A_i \rightarrow M\}_{i\in I}$ and $\{B_j \rightarrow M\}_{j\in J}$ be covering families in $\Mat^\tc$. We want to prove that there exists a common refinement for $A$ and $B$. By Proposition~\ref{prop:indecrefinement}, there exist  indecomposable Tutte coverings $A'$ and $B'$ refining $A$ and $B$, respectively, via  refining morphisms $h_A$ and~$h_B$.
Now, by Proposition~\ref{prop:indareiso}, we have that the multi-sets of matroids $\Ind_M(A')$ and $\Ind_M(B')$ are isomorphic. This means that there exists a bijective map $h\colon \Ind_M(A') \rightarrow \Ind_M(B')$ such that if $h(A'_j)=B'_i$, then $h_j\colon A'_j \rightarrow B'_i$ is an isomorphism of matroids. The composition $h \circ h_B$ is then a refinement morphism and $A'$ refines both $A$ and $B$. This is enough to show that the category $\Mat^\tc$ has refinements.  
\endproof

\section{The $K$-theory of (indecomposable) matroids}
In this section we introduce the  subcategories of  $\Mat^\tc$ of graphic matroids and indecomposable matroids. Our aim is to prove that these categories share the same $K$-theoretic invariants. As a main application, we provide a computation of the $K$-theory of $\Mat^{\tc}$. 

We start with the category with covering families of graphic matroids. Denote by $\Graph$ the full subcategory of $\Mat$ on graphic matroids. Therefore, its objects are matroids of the form $M_G$, for $G$ a finite graph, and the morphisms are the induced morphisms as (graphic) matroids. 
By Proposition~\ref{prop:closuregraphs}, graphic matroids are closed under taking deletions and contractions. Therefore, the  covering family structure on $\Mat_+$ with multi-morphisms the Tutte coverings also yields a category with covering families, that we denote by \[
\Graph^{\tc}\coloneqq(\Graph_+,\mathcal{S}_{\Graph}^{\tc},\ast) \ .
\]
\begin{rmk}
    The inclusion  
    $\Graph^{\tc}\to\Mat^{\tc}$ is a morphism of categories with covering families.
\end{rmk}

We now focus of indecomposable matroids and recall that  an indecomposable matroid is the graphic matroid associated to a forest,  possibly with loops (see Definition~\ref{def:idencmatroids}). We denote by $\LF$ the full subcategory of $\Mat$ on  indecomposable matroids. The morphisms are all matroid morphisms between indecomposable matroids. 

\begin{rmk}\label{rem:IMisminorclosed}
The category $\LF$ is closed under taking minors and duals.
\end{rmk}

Consider  the category~$\LF_+$ of indecomposable matroids with a disjoint distinguished point~$*$. By Example~\ref{ex:catisos}, the isomorphisms of matroids yield a category with covering families that we denote  by 
\[
\LF^{\cong}\coloneqq(\LF_+,\mathcal{S}_{\LF}^{\cong},\ast) \ .
\]
 Observe  that, because of Remark~\ref{rem:IMisminorclosed}, Tutte covering families of  indecomposable matroids yield  a covering family $\mathcal{S}_{\LF}^{\tc}$ in $\LF$, which we still call the Tutte coverings. We get a category with covering families
 \[
 \LF^{\tc}\coloneqq (\LF_+,\mathcal{S}_{\LF}^{\tc},\ast) \ .
 \]
 \begin{rmk}
    The inclusion functors $\LF^{\cong}\to\Mat^{\cong}$, $\LF^{\tc}\to\Graph^{\tc}$ and $\LF^{\tc}\to\Mat^{\tc}$ are all morphisms of categories with covering families. All domain categories are full in the codomain.
\end{rmk}

    Consider the category with covering families $\LF^{\tc}$. By definition, a covering family in $\LF^{\tc}$ is either of the form $\{\ast\to\ast\}_I$ or it is a singleton of the form $\{A\xrightarrow{\cong} C\}$. As a consequence, the category with covering families $\LF^{\tc}$ is in fact equivalent to the category with covering families given by the isomorphisms, that is $\LF^{\cong}$.
As a consequence of this observation we have the following:

\begin{theorem}\label{thm:iso}
    There are  equivalences of spectra
    \[
    K(\iota_1)\colon K(\LF^{\cong})\to K(\Mat^{\tc})
    \]
    and 
    \[
    K(\iota_2)\colon K(\LF^{\cong})\to K(\Graph^{\tc})
    \]
    where $\iota_1$ and $\iota_2$ are the inclusion functors.
\end{theorem}
\begin{proof}
    The category $\Mat^{\tc}$ has refinements by Corollary~\ref{cor:refinements} and $\iota_1\colon \LF^{\cong}\to \Mat^{\tc}$ is the inclusion of a full subcategory with covering families. Every object $M\in \Mat^{\tc}$  has a choice of cover $M_i\to M$ where all $M_i$ are indecomposable matroids by Proposition~\ref{prop:exindec} --  the coverings of the basepoint $\ast$ being the obvious ones, given by  $\{\ast \rightarrow \ast\}$. Then, the first part of the statement follows by Theorem~\ref{thm:devissage}. 
    For the second part of the statement, observe that $\iota_1$ factors through~$\iota_2$. By naturality of the $K$-theory construction, also the second equivalence holds.
\end{proof}

Using this equivalence of spectra, we can now decompose the $K$-theory spectrum of matroids as follows:

\begin{corollary}
    \label{cor:isos}
    There is an equivalence of spectra
    \[
    K(\Mat^{\tc})\simeq \bigvee_{[M]}\Sigma_+^{\infty} B(\mathrm{iso}_{\LF}(M))
    \]
    where the wedge ranges across the isomorphism classes of indecomposable matroids  and $\mathrm{iso}_{\LF}(M)$ is the
group of $\LF$-isomorphisms of  $M$.
\end{corollary}

\begin{proof}
    By Proposition~\ref{prop:isos}, we have the equivalence of spectra
    \[
    K(\LF^{\cong})\simeq \bigvee_{[M]}\Sigma_+^{\infty} B(\mathrm{iso}_{\LF}(M))
    \]
    where the wedge ranges across the isomorphism classes of indecomposable matroids  and the group $\mathrm{iso}_{\LF}(M)$ is the
group of isomorphisms of  $M$ in the category $\LF$. The statement follows by Theorem~\ref{thm:iso}.
\end{proof}

We can provide a more explicit computation, but we first need a preliminary lemma. 

\begin{lemma} \label{lemma:autogr}
Let   $M=(X, \mc{I}) \in \LF^{\cong}$ be an indecomposable matroid of rank~$r$, and let $n$ be the cardinality of the ground set $X$. Then 
 \[ \mrm{iso}_{\LF}(M) \cong S_r \times S_{n-r} \ , \]
 where $S_r$ and $S_{n-r}$ denote the symmetric groups on $r$ and $n-r$ elements.
\end{lemma}
\proof 
 The elements of $\mrm{iso}_{\LF}(M)$ are exactly the automorphisms of $M$ in the category $\Mat$, i.e.~bijective maps $\varphi\colon X \rightarrow X$ such that for any $I \in \mc{I}$ we have that $\varphi (I) \in \mc{I}$. 

Recall that a matroid in $\LF$ has no non-degenerate elements. As a consequence,  $X=X_c \sqcup X_\ell$, where $X_c$ and $X_\ell$ are the sets of coloops and loops of $M$, respectively. Observe that, by definition, $X_c \in \mc{I}$ and it is the unique bases of $M$. In particular we have that $|X_c|=r$ and $|X_\ell|=n-r$.
Furthermore, note that  the condition 
$I \in \mc{I} \Longrightarrow \varphi (I) \in \mc{I}$  is satisfied only if $\varphi({\{x\}}) \in X_i$ for every $x \in X_c$, i.e.~$\varphi$  acts as a permutation of $S_r$ when restricted to $X_c$. Analogously, this condition forces  $\varphi({\{x\}})$ to be in $  X_\ell$ for any $x \in X_\ell$. Then, $\varphi$ acts as an element of $S_{n-r}$ over $X_\ell$. The statement follows.
\endproof 
\begin{rmk}\label{rmk:decindecomposable} We remark that, if $M$ is an indecomposable matroid, then its rank is the number of its coloops and $n-r$ is the number of its loops. In particular $M$ can be written in the form $M\simeq U_{0,1}^{\oplus n-r} \oplus U_{1,1}^{\oplus r}$, where $U_{k,h}$ is the uniform matroid of cardinality $k$ subsets over a set of $h$ elements. 
Furthermore, observe that for any $m,n \in \N$ the matroid $M\simeq U_{0,1}^{\oplus m} \oplus U_{1,1}^{\oplus n}$ is an indecomposable matroid with exactly $m$ loops and $n$ coloops.
\end{rmk}

We observe here that there is an equivalence of spectra
\[
\Sigma_+^{\infty} B(G\times H)\simeq \Sigma_+^{\infty} B(G)\wedge  \Sigma_+^{\infty} B(H)
\]
whenever $G$ and $H$ are finite groups. In fact, first we have an homotopy equivalence of classifying spaces $B(G\times H)\simeq B(G)\times B(H)$.  Further, if $X,Y$ are topological spaces there is a homeomorphism $(X\times Y)_+\cong X_+\wedge Y_+$ of based spaces; see e.g.~\cite[Construction~3.3.14 \& Lemma~3.3.16]{riehl}, applied to the convenient category of compactly generated spaces~$\mathbf{CGTop}$. This implies that the associated category of based spaces~$\mathbf{CGTop}_+$ is closed symmetric monoidal with respect to the smash
product~\cite[Theorem~6.1.12]{riehl}.  Taking the suspension spectrum of a product of based spaces yields a strong symmetric monoidal functor (with respect to the smash product of spectra) -- cf.~\cite[Section~4.8.2]{lurie2017higher}.  As a consequence of this fact, we get the description of the $K$-theory of the category of matroids with Tutte coverings as a wedge decomposition of suspension spectra (cf.~\cite[Example~2.20]{zbMATH07985809}):

\begin{corollary}\label{cor:splitting}
      There are equivalences of spectra
    \[
    K(\Mat^{\tc})\simeq \bigvee_{m,n\in\mathbb{N}} \Sigma_+^{\infty} B(S_m)\wedge \Sigma_+^{\infty} B(S_n)
    \simeq  K(\Graph^{\tc}) \ . \]
\end{corollary}

\begin{proof}
As $K(\Mat^{\tc})$ is equivalent to $K(\LF^{\cong})$, and this is equivalent to $\Sigma_+^{\infty} B(\mathrm{iso}_{\LF}(M))$, we only neeed to specify the automorphism groups 
    $\mathrm{iso}_{\LF}(M)$ for indecomposable matroids. 
  Recall that, by Remark~\ref{rmk:decindecomposable}, an indecomposable matroid can be written in the form $U_{0,1}^{\oplus m} \oplus U_{1,1}^{\oplus n}$, where $m$ and $n$ are the number of loops and coloops, respectively.
    Hence, we get an equivalence of spectra 
       \[
    K(\Mat^{\tc}) \simeq \bigvee_{m,n\in\mathbb{N}}\Sigma_+^{\infty} B(S_m\times S_n) \ .
    \]
 Now, using the   decomposition
    \[
\Sigma_+^{\infty} B(S_m\times S_n)\simeq \Sigma_+^{\infty} B(S_m)\wedge  \Sigma_+^{\infty} B(S_n) \ ,
\]
the result follows.
\end{proof}

As a consequence of the splitting in $K$-theory of Corollary~\ref{cor:splitting}, we get that $K_0(\Mat^{\tc})$ is isomorphic to $\bigoplus_{m,n\in \mathbb{N}}\mathbb{Z}$, as established in \cite{lopez2025realizing}, and used to realise the universal Tutte-Grothendieck invariant
for matroids as the $K_0$-homomorphism induced by a map of $K$-theory spectra. From the same splitting we also have that all higher homotopy groups  encode the stable homology groups of the symmetric groups. Furthermore, we can depict the relation between the $K$-theory of $\Mat^{\cong}$ and $\Mat^{\tc}$ in the following commutative diagram:
\begin{center}
    \begin{tikzcd}
 K(\Mat^{\cong})\arrow[r,""] & K(\Mat^{\tc})\\
 K(\LF^{\cong})\arrow[u,""]\arrow[r, "\cong"]  \arrow[ur,dashed, "\simeq"]   &K(\LF^{\tc})\arrow[u,"\simeq"]
 \end{tikzcd}  
\end{center}
The induced map in $\pi_0$, in the upper part of the diagram, was studied in \cite{lopez2025realizing} to realise the universal Tutte-Grothendieck invariant
for matroids. We ask the following  general question:

\begin{question}
    Is there a categorical interpretation of evaluations of the Tutte polynomial?
\end{question}

We conclude with the description of the $K$-theory of matroids with a further structure reflecting duality of matroids. Let $C_2$ be the cyclic group with two elements:

\begin{corollary}\label{cor:dual}
    The $K$-theory spectra $K(\LF^{\cong})$ and $K(\Mat^{\tc})$ are equivalent as $C_2$-spectra. 
\end{corollary}

\begin{proof}
    Consider the category $\Mat$. This is a category with duality, the duality being given by taking dual matroids. Also the category $\LF$ is a category with duality by Remark~\ref{rem:IMisminorclosed}. Since an element $x$ of a matroid $M$ is non-degenerate if and only if $x$ is non-degenerate in the dual matroid $M^*$, it yields duality functors
    \[
    D_{\Mat^{\tc}}\colon \Mat^{\tc}\to\Mat^{\tc} \text{ and }     D_{\LF^{\cong}}\colon \LF^{\cong}\to\LF^{\cong}
    \]
    which are  morphisms of categories with covering families. That is, the category with Tutte coverings $\Mat^{\tc}$ and the category $\LF^{\cong}$ are $C_2$-categories with covering families (cf.~\cite[Definition~4.1]{zbMATH07985809}). Note that the duality functors commute with the inclusion functor $   \iota_1\colon \LF^{\cong}\to \Mat^{\tc}$, that is the diagram 
    \begin{equation}\label{eq:diag}
 \begin{tikzcd}
 \LF^{\cong}\arrow[r,"\iota_1"]\arrow[d,"D_{\LF^{\cong}}"'] & \Mat^{\tc}\arrow[d,"D_{\Mat^{\tc}}"]\\
 \LF^{\cong}\arrow[r, "\iota_1"]   &\Mat^{\tc}
 \end{tikzcd}    
  \end{equation}
 is commutative. We now focus on the left part of this diagram. The duality functor $D_{\LF^{\cong}}$ induces an equivalence
    \[
    \Sigma_+^{\infty} B(S_m)\wedge \Sigma_+^{\infty} B(S_n)\to \Sigma_+^{\infty} B(S_n)\wedge \Sigma_+^{\infty} B(S_m) 
    \]
    for all $m,n\in \mathbb{N}$, and  it yields a spectrum involution. This  extends to the whole spectrum  $\bigvee_{m,n\in\mathbb{N}} \Sigma_+^{\infty} B(S_m)\wedge \Sigma_+^{\infty} B(S_n)$ which inherits the structure of a $C_2$-spectrum. Likewise, also $K(\Mat^{\tc})$ inherits the structure of a $C_2$-spectrum from the duality functor $D_{\Mat^{\tc}}$. By naturality of the $K$-theory construction applied to the  commutative diagram~\eqref{eq:diag}, the $C_2$-actions are coherent and the equivalence of Theorem~\ref{thm:iso} can be promoted to an equivalence of $C_2$-spectra.
\end{proof}
\bibliographystyle{alpha}
\bibliography{bibliography}
\end{document}